\documentclass[letterpaper, 10 pt, conference]{ieeeconf}  

\IEEEoverridecommandlockouts                              
\title{\LARGE \bf
Geometric Convergence of Distributed Gradient Play in Games with Unconstrained Action Sets}

\author{Tatiana Tatarenko and Angelia Nedi\'c
\thanks{Tatiana Tatarenko is with the Control Methods and Robotics Lab,
        TU darmstadt, Germany.}
\thanks{Angelia Nedi\'c is with School of Electrical, Computer and Energy Engineering,
        Arizona State University, USA.
        }
\thanks{The work has been partially supported by Office of Naval Research grant no. N00014-12-1-0998.}
}

\usepackage{amsmath, amssymb, epsfig, graphicx, bm}
\usepackage{cases}
\usepackage{float, subfigure}
\usepackage{color}
\usepackage{amsthm}
\usepackage{amsfonts}
\usepackage{psfrag}
\usepackage{cite, url}
\usepackage{breqn}
\usepackage[all,cmtip]{xy}
\usepackage{color,hyperref}
\definecolor{darkblue}{rgb}{0,0,1}
\hypersetup{colorlinks,breaklinks,
linkcolor=darkblue,urlcolor=darkblue,anchorcolor=darkblue,citecolor=darkblue}

\usepackage[normalem]{ulem} 

\newcommand{\smax}[1]{{\sigma_{\max}\{#1\}}}

\newcommand{\lminnz}[1]{{\tilde{\lambda}_{\min}\{#1\}}}

\newtheorem{theorem}{Theorem}
\newtheorem{definition}{Definition}

\newtheorem{lemma}{Lemma}
\newtheorem{remark}{Remark}

\newtheorem{assumption}{Assumption}

\newcommand{\R}{{\mathbb{R}}}
\newcommand{\bx}{{\mathbf{x}}}
\newcommand{\bbx}{{\bar{\mathbf{x}}}}

\newcommand{\by}{{\mathbf{y}}}

\newcommand{\bz}{{\mathbf{z}}}

\newcommand{\Diag}{{\mathrm{Diag}}}
\newcommand{\di}{{\mathrm{diag}}}

\newcommand{\A}{{\mathcal{A}}}

\newcommand{\Gra}{{\mathcal{G}}}

\newcommand{\tx}{{\tilde{x}}}

\newcommand{\Om}{{\Omega}}

\newcommand{\one}{{\mathbf{1}}}

\newcommand{\N}{{\mathcal{N}}}

\newcommand{\bF}{{\mathbf{F}}}
\newcommand{\tbF}{{\tilde{\mathbf{F}}}}

\newcommand{\Fro}{{\mathrm{Fro}}}

\newcommand{\trace}{{\mathrm{trace}}}
\newcommand{\T}{{\mathrm{T}}}

\begin{document}
\maketitle

\begin{abstract}                          
We provide a distributed algorithm to learn a Nash equilibrium in a class of non-cooperative games with strongly monotone mappings and unconstrained action sets. Each player has access to her own smooth local cost function and can communicate to her neighbors in some undirected graph. We consider a distributed communication-based gradient algorithm. 
For this procedure, we prove geometric convergence to a Nash equilibrium. 
In contrast to our previous works \cite{Cdc2018_TatShiNed, GRANE}, where the proposed algorithms required two parameters to be set up and the analysis was based on a so called augmented game mapping, the procedure in this work corresponds to a standard distributed gradient play and, thus, only one constant step size parameter needs to be chosen appropriately to guarantee fast convergence to a game solution. 
Moreover, we provide a rigorous comparison between the convergence rate of the proposed distributed gradient play and the rate of the GRANE algorithm presented in \cite{GRANE}. It allows us to demonstrate that the distributed gradient play outperforms the GRANE in terms of convergence speed.
\end{abstract}

\section{Introduction}
There are a lot of multi-agent systems, where agents' objective functions are coupled through decision variables of all agents in a system. In such cases, game theory is a useful tool to deal with the corresponding optimization problems. The applications of game theory in engineering can be found, for example, in electricity markets, power systems,  flow control problems, and communication networks \cite{Alpcan2005, BasharSG, Scutaricdma}. Desirable outcomes in games are characterized by so called Nash equilibria, which correspond to a stable state from which no agent has motivation to deviate. This paper provides a distributed discrete-time algorithm applicable to fast Nash equilibrium seeking in a class of non-cooperative games under the assumption that agents exchange their local information with neighbors by means of some communication topology. 

Distributed communication-based algorithms are proposed for \textit{aggregative games} \cite{Koshal2012, paccagnan2016aggregative}. Communication protocols are applied to different classes of games with some convergence guarantees \cite{salehisadaghiani2016distributed, ADMM_WS2017, TatAut19}. The work \cite{salehisadaghiani2016distributed}  proposes a gradient based gossip algorithm to learn Nash equilibria in games. Under some technical assumptions, this algorithm converges almost surely to the Nash equilibrium, given a diminishing step size.  Under further assumption of strong convexity, with some constant step size $\alpha$, the algorithm converges to an $O(\alpha)$ neighborhood of the Nash equilibrium in average. The work \cite{ADMM_WS2017} develops an algorithm within the framework of inexact-ADMM and proves its convergence to the Nash equilibrium with the rate $o(1/k)$ under cocoercivity of the game mapping. However, no aforementioned work aims to provide algorithms which converge to a Nash equilibrium 
with a fast geometric rate.

The paper \cite{Cdc2018_TatShiNed} leverages the idea of an accelerated approach for solving variational inequalities \cite{Nesterov} and provides a version of the gradient play algorithm (Acc-GRANE) that guaratees a fast convergence to the Nash equilibrium with an explicitly good dependence on the condition number. The analysis is based on strong monotone properties of a so called augmented mapping which takes into account not only the gradients of the cost functions, but also the communication settings. The presented algorithm is applicable only to a sub-class of games characterized by a restrictive connection between the number of players, Lipschitz constant, and parameter of strong monotonicity.
To apply the distributed gradient play algorithm to a broader class of games, work \cite{GRANE} considers the case of the restricted strongly monotone augmented game mapping and demonstrates geometric convergence of the procedure to the Nash equilibrium. However, both types of the procedures mentioned above require a careful set up not only for the step size parameter but also for the augmented mapping. In this paper we provide a distributed gradient play whose convergence properties are not based on the augmented mapping. This fact allows us to focus only on the choice of the step size in the optimization procedure. Moreover, a rigorous comparison between the convergence rates of the proposed distributed gradient play and the GRANE in \cite{GRANE} demonstrates that the algorithm presented in this paper converges faster to a Nash equilibrium under the same settings of a game.

This paper is organized as follows. In Section \ref{sec:problem},  we set up the game under consideration. In Section \ref{sec:main}, we present the distributed gradient play algorithm to seek its solutions. In Section \ref{sec:proof}, we prove the main result stating a geometric convergence of the proposed procedure. Section \ref{sec:algorithms} compares the convergence rates of the proposed algorithm and the GRANE from work \cite{GRANE}. We provide a numerical case study in Section \ref{sec:sim}. In Section \ref{sec:conclusion}, we summarize the result and discuss future work.

\textbf{Notations.}
The set $\{1,\ldots,n\}$ is denoted by $[n]$.
For any function $f:K\to\R$, $K\subseteq\R^n$, $\nabla_i f(x) = \frac{\partial f(x)}{\partial x_i}$ is the partial derivative taken in respect to the $i$th coordinate of the vector variable $x\in\R^n$.
For any real vector space $\tilde E$ its dual space is denoted by $\tilde E^*$ and the inner product is denoted by $\langle u,v \rangle$, $u\in\tilde E^*$, $v\in \tilde E$. An operator $B:\tilde E\to\tilde E^*$ is positive definite if $\langle Bv,v \rangle>0$ for all $v\in\tilde E\setminus \{0\}$. An operator $B:\tilde E\to\tilde E^*$ is self-adjoint if $\langle Bv,v' \rangle = \langle Bv',v \rangle$ for all $v',v\in\tilde E$. 
Given a positive definite and self-adjoint operator $B$, we define the Euclidean norm on $\tilde E$ induced by $B$ as $\|v\| = \langle Bv,v \rangle^{1/2}$.
Any mapping $g:\tilde E\to \tilde E^*$ is said to be \emph{strongly monotone with the constant} $\mu>0$ on $Q\subseteq \tilde E$, if $\langle g(u)-g(v), u - v \rangle\ge\mu\|u - v\|^2$ for any $u,v\in Q$, where $\|\cdot\|$ is the corresponding norm in $\tilde E$.
We consider real vector space $E$, which is either space of real vectors $E = E^* = \R^n$ or the space of real matrices $E = E^* = \R^{n\times n}$. In the case $E = \R^{n\times n}$ the inner product $\langle u,v \rangle \triangleq \sqrt{\trace(u^Tv)}$ is the Frobenius inner product on $\R^{n\times n}$.
In the case $E = \R^n$ we use $\|\cdot\|$ to denote the Euclidean norm induced by the standard dot product in $\R^n$, whereas  in the case $E = \R^{n\times n}$ we use $\|\cdot\|_{\Fro}$ to denote the Frobenius norm induced by the Frobenius inner product i.e. $\|v\|_{\Fro} \triangleq \sqrt{\trace(v^Tv)}$. 
 The largest singular value of a matrix $A$ is denoted by
 $\smax{A}$.
 The smallest \emph{nonzero} eigenvalue of a positive semidefinite matrix $A\not=0$ is denoted by
 $\lminnz{A}$, which is strictly positive. 
For any matrix $A\in\R^{n\times n}$ we use $\di(A)$ to denote its diagonal vector, i.e. $\di(A) = (a_{11},\ldots, a_{nn})$.
For any vector $a\in\R^n$ we  use $\Diag(a)$ to denote the diagonal matrix with the vector $a$ on its diagonal.
We call a matrix $A$ \emph{consensual}, if it has equal row vectors.

\section{Problem Formulation}\label{sec:problem}
We consider a non-cooperative game between $n$ players with unconstrained action sets. Let $J_i$ and $\Om_i = \R$ denote respectively the cost function and the feasible action set of the player $i$\footnote{All results below are applicable for games with different dimensions $\{d_i\}$ of the action sets, i.e., $\Omega_i=R^{d_i}$ for all $i$. The one-dimensional case is considered for the sake of notation simplicity.}. Each function $J_i(x_i,x_{-i})$, $i\in[n]$, depends on $x_i$ and $x_{-i}$, where $x_i\in\R$ is the action of the player $i$ and $x_{-i}\in\Om_{-i}=\R^{n-1}$ denotes the joint action of all players except for the player $i$. 
Overall in this paper, we assume that the cost function $J_i(x_i, x_{-i})$ is continuously differentiable in $x_i$ for each fixed $x_{-i}$, $i\in[n]$. Then we define the game mapping as 
\begin{align}\label{eq:gamemapping}
&\bF(x)\triangleq\left[\nabla_1 J_1(x_1,x_{-1}), \ldots, \nabla_n J_n(x_n,x_{-n})\right]^T,
 \end{align}
 where $\nabla_i J_i(x_i,x_{-i}) = \frac{\partial J_i(x_i,x_{-i})}{\partial x_i}$ for all $i\in[n]$.
We assume that the players can interact over an undirected communication graph $\Gra([n],\A)$. The set of nodes is the set of the player $[n]$ and the set of undirected arcs $\A$ is such that $(i,j)\in\A$ if and only if $(j,i)\in\A$, i.e. there is an undirected communication link between $i$ to $j$. 
Thus, some information (message) can be passed from the player $i$ to the player $j$ and vice versa. For each player $i$ the set $\N_i$ is the set of neighbors in the graph $\Gra([n],\A)$, namely $\N_{i}\triangleq\{j\in[n]: \, (i,j)\in\A\}$.
Let us denote the game introduced above by $\Gamma(n,\{J_i\},\{\Om_i=\R\},\Gra)$.
We make the following assumptions regarding the game $\Gamma$.

\begin{assumption}\label{assum:convex}
 The game mapping $\bF(x)$ is \emph{strongly monotone} on $\R^n$ with the constant $\mu>0$.
\end{assumption}
Note that Assumption~\ref{assum:convex} above implies strong convexity of each cost function $J_i(x_i, x_{-i})$ in $x_i$ for any fixed $x_{-i}$ with the constant $\mu$. 

\begin{assumption}\label{assum:Lipschitz}
Each function $\nabla_i J_i(\cdot):\R^n\to\R$, $i\in[n]$, is Lipschitz continuous on $\R^{n}$, namely for some constant $L_i\ge 0$, we have $\forall\ x, y\in\R^n$
\begin{align*}
 |\nabla_i J_i(x)-\nabla_i J_i(y)|&\leq L_i\|x-y\|.
\end{align*}
\end{assumption}
\begin{remark}
 An example of games satisfying Assumption~\ref{assum:Lipschitz} above is a class of aggregative games \cite{Koshal2012, paccagnan2016aggregative}, where each cost function  $J_i$ is of the following form: 
 \[J_i(x_i,x_{-i}) = c_i(x_i) - x_iU_i(\sum_{j=1}^n x_j).\]
 Here $c_i(\cdot):\R\to\R$ is an agent specific function with a Lipschitz continuous derivative and the \emph{linear} function $U_i(\sum_{j=1}^n x_j)$ captures the utility associated with aggregate output $\sum_{j=1}^n x_j$.
\end{remark}

Two assumptions above are standard for works aiming to demonstrate geometric convergence of algorithms for computing an equilibrium point in variational inequalities and games. 

Finally, we make the following assumption on the communication graph, which guarantees sufficient information  "mixing" in the network.
\begin{assumption}\label{assum:connected}
The underlying undirected communication graph $\Gra([n],\A)$ is connected. There is a non-negative matrix $W=[w_{ij}]\in\R^{n\times n}$ associated with the graph such that $w_{ij}>0$ if and only if $(i,j)\in\A$. Moreover, $W$ is doubly stochastic, i.e. $\sum_{l=1}^{n}w_{lj} = \sum_{l=1}^{n}w_{il} = 1$, $\forall i,j\in[n]$.
\end{assumption}
\begin{remark}
The weight matrix $W$ from Assumption~\ref{assum:connected} need not be symmetric. Some simple strategies that generate
  symmetric mixing matrices for which Assumption~\ref{assum:connected} holds can be found in Section 2.4 in \cite{Shi2014}.
\end{remark}
Assumption~\ref{assum:connected} implies that the second largest singular value $\sigma$ of $W$ is such that $\sigma\in(0,1)$ and for any $x\in\R^n$ the following average property holds (see \cite{OlshTsits}):
\begin{align}\label{eq:avprop}
 \|Wx-\one\bar{x}\|\le \sigma\|x-\one\bar{x}\|,
\end{align}
where $\bar{x} = \frac{1}{n}\one^T{x}$ is the average of the coordinates of $x$.

One of the stable solutions in any game $\Gamma$ corresponds to a Nash equilibrium defined below.
\begin{definition}\label{def:NE}
 A vector $x^*=[x_1^*,x_2^*,\cdots, x_n^*]^T\in\Om$ is a \emph{Nash equilibrium} if for any $i\in[n]$ and $x_i\in \Om_i$
 $$J_i(x_i^*,x_{-i}^*)\le J_i(x_{i},x_{-i}^*).$$
 \end{definition}

In this work, we are interested in \emph{distributed seeking of a Nash equilibrum} in a game $\Gamma(n,\{J_i\},\{\Om_i=\R\},\Gra)$ for which Assumptions~\ref{assum:convex}-\ref{assum:connected} hold.
Note that under Assumption~\ref{assum:convex}, the game $\Gamma(n,\{J_i\},\{\Om_i=\R\},\Gra)$ has a unique Nash equilibrium \cite{Rosen}. Moreover, as $\Om_i=\R$ and $J_i(x_i,x_{-i})$ is strongly monotone in $x_i$ over $\R$ for all $i\in[n]$ the vector $x^*\in\R^n$ is the unique Nash equilibrium if and only if 
\begin{align}\label{eq:NEcond}
\bF(x^*) = \boldsymbol 0. 
\end{align}

\section{Nash Equilibria Learning in Distributed Settings}\label{sec:main}
To deal with the partial information available to players which is exchanged among them over the communication graph, we assume that each player $i$ maintains a \emph{local variable}
\begin{align}\label{eq:vector}
x_{(i)}=[\tx_{(i)1},\cdots,\tx_{(i)i-1},x_i,\tx_{(i)i+1},\cdots,\tx_{(i)n}]^T\in\R^n,
\end{align}
which is her estimation of the joint action $x=[x_1,x_2,\cdots,x_n]^T$.
Here $\tx_{(i)j}\in\R$ is the player $i$'s estimate of $x_j$ and  $\tx_{(i)i}=x_i\in\Om_i=\R$. Also, we compactly denote the estimates of other players' actions by the player $i$ as
\begin{align}\label{eq:vector1}
\tx_{-i}=[\tx_{(i)1},\cdots,\tx_{(i)i-1},\tx_{(i)i+1},\cdots,\tx_{(i)n}]^T\in\R^{n-1},
\end{align}
and the estimates of the player $j$'s action $x_j$ by all players as
$$\tx_{(:)j}=[\tx_{(1)j},\cdots,\tx_{(j-1)j},x_{j},\tx_{(j+1)j},\cdots,\tx_{(n)j}]^T\in\R^n.$$
Thus, we can define the estimation matrix $\bx\in\R^{n\times n}$, where the $i$th row is equal to the estimation vector $x_{(i)}$, $i\in[n]$, namely
$$
  \bx\triangleq\left(
     \begin{array}{ccc}
       \textrm{---}& x_{(1)}^\T & \textrm{---} \\
       \textrm{---}& x_{(2)}^\T & \textrm{---} \\
       &\vdots& \\
       \textrm{---}& x_{(n)}^\T & \textrm{---} \\
     \end{array}
   \right).
$$
For any given estimation matrix, we define the diagonal matrix $\tbF(\bx)\in\R^{n\times n}$ with $\tbF(\bx)_{ii}= \nabla_i J_i(x_{(i)})$, $i\in[n]$, namely
\begin{align}\label{eq:diaggrad}
 \tbF(\bx)\triangleq \Diag(\nabla_1 J_1(x_{(1)}),\ldots,\nabla_n J_n(x_{(n)})).
\end{align}

We propose the following distributed gradient play procedure for learning a Nash equilibrium in the game $\Gamma(n,\{J_i\},\{\Om_i=\R\},\Gra)$. 
According to this algorithm, each player $i$ updates its local estimation of the joint action as follows:

\begin{align*}
 &x_{i}^{t+1} = \sum_{j=1}^n w_{ij}x_{(j)i}^{t} - \alpha\nabla_{x_i}J_i(x_{(i)}^{t}), \cr
 &x_{(i)l}^{t+1}= \sum_{j=1}^n w_{ij}x_{(j)l}^{t}, \quad \mbox{for }l\ne i, \, i\in[n].
\end{align*}
Thus, to get a new estimation $x_{(i)}^{t+1}$ each agent $i$ aggregates over the communication graph the current estimations of its neighbors and, makes a local gradient step with a step size $\alpha$ in respect to the gradient of its cost function $\nabla_{x_i}J_i(x_{(i)}^{t})$ calculated at the current local estimation  $x_{(i)}^{t}$.
The local updates above can be represented in the following compact vector form:
\begin{align}\label{eq:alg}
 \bx^{t+1} = W\bx^{t} - \alpha\tbF(\bx^{t}), 
\end{align}
where $\alpha$ is a constant step size to be set up.

In the following we prove geometric convergence of the procedure above to the unique Nash equilibrium in the game $\Gamma(n,\{J_i\},\{\Om_i=\R\},\Gra)$ under Assumptions~\ref{assum:convex}-\ref{assum:connected} and an appropriate choice of $\alpha$. This result is formulated in the following theorem. 


\begin{theorem}\label{th:main}
  Let $\Gamma(n,\{J_i\},\{\Om_i=\R\},\Gra)$ be a game for which Assumptions~\ref{assum:convex}-\ref{assum:connected} hold. Let $\mu$ and $\sigma$ be as defined in Assumption~\ref{assum:convex} and relation \eqref{eq:avprop}, respectively, and the step size parameter $\alpha$ be chosen as follows: 
  \begin{equation*}
  \begin{split} 
     0<\alpha<\min&\left\{1,\frac{\mu}{2L^2}, \frac{\sigma}{2L}\frac{\sqrt {n}}{\sqrt{n-1}}\left(\frac{\sqrt 2}{\sqrt{1+\sigma^2}}-1\right),\right.\\
   &\left.\qquad\frac{n}{\mu}\left(\frac{8}{(\sqrt{1+\sigma^2}-\sqrt 2)^2}-1\right),\right.\\
   &\left.\qquad\qquad\quad\frac{\sqrt{n^2+\frac{2\mu^4(1-\sigma^2)}{(n-1)L^4(1+\sigma^2)}}-n}{2\mu}\right\},
   \end{split}
 \end{equation*}
 where $L$ is the Lipschitz constant defined in Assumption~\ref{assum:Lipschitz}.
  Then, the algorithm \eqref{eq:alg} converges to the consensual matrix $\bx^*$ whose rows are equal to the row-vector $x^*$ which is the unique Nash equilibrium in the game $\Gamma(n,\{J_i\},\{\Om_i=\R\},\Gra)$. Moreover,
  \[\|\bx^{t}-\bx^*\|^2_{\Fro}\le O(q^t)\]
  for some $q = q(\alpha)\in(0,1)$. 
  \end{theorem}

In the next section we provide the proof of the main result formulated in Theorem~\ref{th:main} above.

\section{Proof of Main Result}\label{sec:proof}
Let $\bbx^{t}$ be a consensual matrix with the rows equal to the vector $\bar{x}^{t} = \frac{1}{n}\sum_{i=1}^n x^t                                                                                                                                                                                                                                                  _{(i)}$. This matrix corresponds to the running average of the players' estimations of the current joint action. Under Assumption~\ref{assum:connected}, as seen from the definition of the algorithm in \eqref{eq:alg}, the average $\bar x^t$ evolves according to the following relation:
  \begin{align}\label{eq:runav}
   \bbx^{t+1} = \bbx^{t} - \frac{\alpha}{n}\bF^0(\bx^{t}),
  \end{align}
  where $\bF^0(\cdot)$ is the consensual matrix with the rows equal to $\di(\tbF(\cdot))$.

 To prove Theorem~\ref{th:main}, we start by proving some lemmata. 
 First of all, we estimate the consensus violation term in respect to the running average, namely $\|\bx^{t+1} - \bbx^{t+1}\|_{\Fro}$.
 \begin{lemma}\label{lem:distToRunav}
  Under Assumption~\ref{assum:connected} the following holds for the procedure \eqref{eq:alg}: 
  \[\|\bx^{t+1} - \bbx^{t+1}\|_{\Fro}\le\sigma\|\bx^{t} - \bbx^{t}\|_{\Fro} + \alpha\frac{\sqrt {n-1}}{\sqrt n}\|\tbF(\bx^{t})\|_{\Fro}.\]
 \end{lemma}
 \begin{proof}
 Taking into account \eqref{eq:avprop} and \eqref{eq:runav}, we conclude that
  \begin{align*}
   \|\bx^{t+1} &- \bbx^{t+1}\|_{\Fro} \cr
   &= \|W\bx^{t} - \alpha\tbF(\bx^{t}) - \bbx^{t} + \frac{\alpha}{n}\bF^0(\bx^{t})\|_{\Fro}\cr
   &\le \sigma\|\bx^{t} - \bbx^{t}\|_{\Fro} + \frac{\alpha}{n}\|n\tbF(\bx^{t})-\bF^0(\bx^{t})\|_{\Fro}.
  \end{align*}
 Next, note that 
 \begin{align*}
  &n\tbF(\bx^{t})-\bF^0(\bx^{t}) \cr
  &= \begin{pmatrix} 
(n-1)\nabla_1 J_1(x_{(1)}^{t})  & \ldots & -\nabla_n J_n(x_{(n)}^{t})\\
-\nabla_1 J_1(x_{(1)}^{t})  & \ldots & -\nabla_n J_n(x_{(n)}^{t})\\
\vdots & \ddots & \vdots & \\
-\nabla_1 J_1(x_{(1)}^{t})  & \ldots & -(n-1)\nabla_n J_n(x_{(n)}^{t})
\end{pmatrix} 
 \end{align*}
and, hence, 
\begin{align*}\|n\tbF(\bx^{t})-\bF^0(\bx^{t})\|_{\Fro} &= \sqrt{n(n-1)\sum_{i=1}^n(\nabla_i J_i(x_{(i)}^{t}))^2}\cr
 & = \sqrt{n(n-1)}\|\tbF(\bx^{t})\|_{\Fro},
\end{align*}
thus yielding the stated result. 
 \end{proof}
 Next, to estimate the distance $\|\bx^{t+1} - \bbx^{t+1}\|_{\Fro}$ in terms of optimum violation $\|\bx^{t}-\bx^*\|_{\Fro}$, we upper bound $\|\tbF(\bx^{t})\|_{\Fro}$ in the following lemma.
 
\begin{lemma}\label{lem:grad}
 Let Assumption~\ref{assum:Lipschitz} hold in the game $\Gamma(n,\{J_i\},\{\Om_i=\R\},\Gra)$. Then  
 \[\|\tbF(\bx^{t})\|_{\Fro}\le L\|\bx^{t}-\bx^*\|_{\Fro},\]
 where $L = \max_{i} L_i$ and $L_i$, $i\in[n]$, are the Lipschitz constants from Assumption~\ref{assum:Lipschitz}.
\end{lemma}
\begin{proof}
 Due to Assumption~\ref{assum:Lipschitz} and the fact that $\bF(\bx^*) = \boldsymbol 0$ (see \eqref{eq:NEcond}),
 \begin{align*}\|\tbF(\bx^{t})\|_{\Fro} &= \|\tbF(\bx^{t})-\tbF(\bx^*)\|_{\Fro}\cr
 &=\sqrt{\sum_{i=1}^n(\nabla_i J_i(x_{(i)}^{t})-\nabla_i J_i(x^*))^2}\cr
 &=\sqrt{\sum_{i=1}^nL_i^2\|x_{(i)}^{t}-x^*\|^2} \le L\sqrt{\sum_{i=1}^n\|x_{(i)}^{t}-x^*\|^2} \cr
 &= L\|\bx^{t}-\bx^*\|_{\Fro}.
 \end{align*}
\end{proof}

Finally, we analyze the distance between the running average and the Nash equilibrium. 
\begin{lemma}\label{lem:distToNE}
 Let Assumptions~\ref{assum:convex}-\ref{assum:connected} hold in the game $\Gamma(n,\{J_i\},\{\Om_i=\R\},\Gra)$.
 Then for any $\theta>0$ and the step size $\alpha\le\frac{\theta}{L^2}$ the following inequality holds:
  \begin{align*}
  \left(1+\frac{2\alpha}{n}\left(\mu-\frac{\theta}{2}\right)\right)&\|\bbx^{t+1}-\bx^*\|^2_{\Fro} \cr
  &\le\|\bbx^{t}-\bx^*\|^2_{\Fro} + \frac{L^2\alpha}{\theta}\|\bx^{t}-\bbx^{t}\|^2_{\Fro}.
  \end{align*}
\end{lemma}
\begin{proof}
Let $\hat{\bF} (\bx^t) =(\nabla_1 J_1(x_{(1)}^{t}),\ldots, \nabla_n J_n(x_{(n)}^{t}) )^T\in\R^{n} $.
 Using the equality $\bar{x}^{t+1} = \bar{x}^{t} - \frac{\alpha}{n}\hat{\bF}(\bx^{t})$ (see \eqref{eq:runav}) and the basic inequality $\|a\|^2 = \|a+b\|^2 - 2\langle a,b \rangle - \|b\|^2$ for $a = \bar{x}^{t+1}-x^*$ and $b = \bar{x}^{t} - \bar{x}^{t+1}$, we obtain
 \begin{align}\label{eq:eq1}
  \|\bar{x}^{t+1}-x^*\|^2& = \|\bar{x}^{t}-x^*\|^2 - 2\langle \bar{x}^{t+1}-x^*,\bar{x}^{t} - \bar{x}^{t+1} \rangle \cr
  &\qquad\qquad\qquad- \|\bar{x}^{t} - \bar{x}^{t+1}\|^2\cr
  &= \|\bar{x}^{t}-x^*\|^2  - \frac{2\alpha}{n} \langle \bar{x}^{t+1}-x^*, \hat{\bF}(\bx^{t})\rangle \cr
  &\qquad\qquad\qquad- \|\bar{x}^{t} - \bar{x}^{t+1}\|^2.
 \end{align}
We proceed with estimating the term $\langle \bar{x}^{t+1}-x^*, \hat{\bF}(\bx^{t})\rangle$.
\begin{align}\label{eq:eq2}
 \langle \bar{x}^{t+1}-x^*, \hat{\bF}(\bx^{t})\rangle & = \langle \hat{\bF}(\bx^{t}) - \bF(\bar{x}^{t+1}),\bar{x}^{t+1}-x^* \rangle  \cr
 & \quad + \langle \bF(\bar{x}^{t+1}) - \bF(x^*),\bar{x}^{t+1}-x^* \rangle\cr
 &\ge\langle \hat{\bF}(\bx^{t}) - \bF(\bar{x}^{t+1}),\bar{x}^{t+1}-x^* \rangle\cr
 &\quad + \mu\|\bar{x}^{t+1}-x^*\|^2,
\end{align}
where in the first equality we used the fact that $\bF(x^*)= \boldsymbol 0$ (see \eqref{eq:NEcond}) and in the last inequality we used Assumption~\ref{assum:convex}.
Next, for any $\theta>0$ we obtain
\begin{align}\label{eq:eq3}
 &\langle \hat{\bF}(\bx^{t})  - \bF(\bar{x}^{t+1}),\bar{x}^{t+1}-x^* \rangle\ge -\frac{\theta}{2}\|\bar{x}^{t+1}-x^*\|^2 \cr
 &\qquad\qquad\qquad\qquad\qquad\qquad- \frac{1}{2\theta}\|\hat{\bF}(\bx^{t}) - \bF(\bar{x}^{t+1})\|^2\cr
 &= -\frac{\theta}{2}\|\bar{x}^{t+1}-x^*\|^2 - \frac{1}{2\theta}\sum_{i=1}^n \|\nabla_i J_i(x_{(i)}^{t}) - \nabla_i J_i(\bar{x}^{t+1})\|^2\cr
 &\ge -\frac{\theta}{2}\|\bar{x}^{t+1}-x^*\|^2 - \frac{L^2}{2\theta}\sum_{i=1}^n \|x_{(i)}^{t} - \bar{x}^{t+1}\|^2.
\end{align}
Bringing \eqref{eq:eq2} and \eqref{eq:eq3} into \eqref{eq:eq1} 
we conclude 
that
\begin{align*}
  \|\bar{x}^{t+1}-x^*\|^2& \le \|\bar{x}^{t}-x^*\|^2 - \frac{2\alpha}{n}\left(\mu-\frac{\theta}{2}\right)\| \bar{x}^{t+1}-x^*\|^2 \cr
  &+\frac{\alpha L^2}{n\theta}\sum_{i=1}^n \|x_{(i)}^{t} - \bar{x}^{t+1}\|^2-  \|\bar{x}^{t} - \bar{x}^{t+1}\|^2.
 \end{align*}
Further, taking into account that
$\sum_{i=1}^n \|x_{(i)}^{t} - \bar{x}^{t+1}\|^2 = \sum_{i=1}^n \|x_{(i)}^{t} - \bar{x}^{t}\|^2 + n\|\bar{x}^{t} - \bar{x}^{t+1}\|^2$, we see that
\begin{align}\label{eq:eq4}
  \|\bar{x}^{t+1}&-x^*\|^2 \le \|\bar{x}^{t}-x^*\|^2 - \frac{2\alpha}{n}\left(\mu-\frac{\theta}{2}\right)\| \bar{x}^{t+1}-x^*\|^2 \cr
  &+\frac{\alpha L^2}{n\theta}\sum_{i=1}^n \|x_{(i)}^{t} - \bar{x}^{t}\|^2 + (\frac{\alpha L^2}{\theta}-1)\|\bar{x}^{t} - \bar{x}^{t+1}\|^2.
 \end{align}
Next, taking into that $\alpha<\frac{\theta}{L^2}$,
$\sum_{i=1}^n \|x_{(i)}^{t} - \bar{x}^{t}\|^2 = \|\bx^t - \bbx^{t}\|^2_{\Fro}$, and
that for any consensual matrices $\bx\in\R^{n\times n}$, $\by\in\R^{n\times n}$ with the vectors $x,y\in\R^{n}$ as their rows respectively we have 
$\|x-y\|^2 = \frac{1}{n}\|\bx-\by\|^2_{\Fro}$, we get from \eqref{eq:eq4} 
\begin{align*}
 &\left(1+\frac{2\alpha}{n}\left(\mu-\frac{\theta}{2}\right)\right)\|\bbx^{t+1}-\bx^*\|^2_{\Fro}  \cr
 &\le \|\bbx^{t}-\bx^*\|^2_{\Fro} + \frac{L^2\alpha}{\theta}\|\bx^{t}-\bbx^{t}\|^2_{\Fro}.
\end{align*}
\end{proof}

Having these three lemmata in place, we are ready to prove the main result.
 \begin{proof}[Proof of Theorem~\ref{th:main}]
Taking into account Lemma~\ref{lem:distToRunav} and Lemma~\ref{lem:grad}, we conclude that,  under conditions of the theorem, we have
 \begin{align*}
  \|\bx^{t+1} - \bbx^{t+1}\|_{\Fro}&\le\sigma\|\bx^{t} - \bbx^{t}\|_{\Fro} + \alpha\frac{\sqrt {n-1}}{\sqrt{n}}\|\tbF(\bx^{t})\|_{\Fro},\cr
  \|\tbF(\bx^{t})\|_{\Fro}&\le L\|\bx^{t}-\bx^*\|_{\Fro}.\cr
 \end{align*}
The inequalities above together with 
\[\|\bx^{t}-\bx^*\|_{\Fro}\le\|\bx^{t}-\bbx^t\|_{\Fro}+\|\bbx^{t}-\bx^*\|_{\Fro}\]
imply that 
\begin{align}\label{eq:ineq1}
 \|\bx^{t+1} - \bbx^{t+1}\|_{\Fro}&\le \sigma\|\bx^{t} - \bbx^{t}\|_{\Fro} \cr
 &\qquad+ \alpha \frac{\sqrt {n-1}}{\sqrt{n}}L\|\bx^{t}-\bx^*\|_{\Fro}\cr
 & \le(\sigma+\alpha\frac{\sqrt {n-1}}{\sqrt{n}} L)\|\bx^{t} - \bbx^{t}\|_{\Fro} \cr
 &\qquad+ \alpha\frac{\sqrt {n-1}}{\sqrt{n}} L \|\bbx^{t}-\bx^*\|_{\Fro}.
\end{align}
Next, we apply to \eqref{eq:ineq1} the standard inequality $(a+b)^2 \le (1+\beta)a^2 + \frac{1+\beta}{\beta}b^2$, which holds for any real numbers $a,b\in\R$ and any $\beta>0$.
By taking $a = (\sigma+\alpha \frac{\sqrt {n-1}}{\sqrt{n}}L)\|\bx^{t} - \bbx^{t}\|_{\Fro}$ and $b = \alpha \frac{\sqrt {n-1}}{\sqrt{n}}L \|\bbx^{t}-\bx^*\|_{\Fro}$, we get
\begin{align}\label{eq:ineq2}
 \|\bx^{t+1} - \bbx^{t+1}\|&^2_{\Fro}\le ((\sigma+\alpha \frac{\sqrt {n-1}}{\sqrt{n}}L)\|\bx^{t} - \bbx^{t}\|_{\Fro} \cr
 &\qquad\qquad+ \alpha\frac{\sqrt {n-1}}{\sqrt{n}} L \|\bbx^{t}-\bx^*\|_{\Fro})^2\cr
 &\le (1+\beta)(\sigma+\alpha\frac{\sqrt {n-1}}{\sqrt{n}} L)^2\|\bx^{t} - \bbx^{t}\|^2_{\Fro} \cr
 &\quad+ \frac{1+\beta}{\beta}\alpha^2 \frac{n-1}{n}L^2 \|\bbx^{t}-\bx^*\|^2_{\Fro}.
\end{align}
Moreover, Lemma~\ref{lem:distToNE} with the choice $\theta=\mu$ implies
\begin{align}\label{eq:ineq3}
 \|\bbx^{t+1}-\bx^*\|^2_{\Fro}\le \gamma\|\bbx^{t}-\bx^*\|^2_{\Fro} + \gamma\frac{2L^2\alpha}{\mu}\|\bx^{t}-\bbx^{t}\|^2_{\Fro},
\end{align}
where $\gamma = \frac{1}{1+\frac{\mu\alpha}{n}}$.
Let $\bz^t = (\|\bbx^{t}-\bx^*\|^2_{\Fro}, \|\bx^{t} - \bbx^{t}\|^2_{\Fro})$. Then taking \eqref{eq:ineq2} and \eqref{eq:ineq3} into account, we conclude that 
\begin{align}\label{eq:dynamics}
 \bz^{t+1} \le Z(\alpha, \mu,\beta)\bz^t,
\end{align}
where 
\begin{align*}
 Z = \begin{pmatrix} 
\gamma &  \gamma\frac{2L^2\alpha}{\mu}\\
\frac{1+\beta}{\beta}\frac{n-1}{n}\alpha^2 L^2 & (1+\beta)(\sigma+\alpha \frac{\sqrt {n-1}}{\sqrt{n}}L)^2
\end{pmatrix}.
\end{align*}
We proceed with analysis of the properties of the positive matrix $Z=Z(\alpha, \mu,\beta)$. 
First, we calculate its eigenvalues. Its characteristic polynomial is 
\begin{align*}
 p_{Z}(\lambda) = \left(\lambda - \gamma\right)&\left(\lambda - (1+\beta)(\sigma+\alpha\frac{\sqrt {n-1}}{\sqrt{n}} L)^2\right) \cr
 &- \gamma\frac{n-1}{n}\frac{2\alpha^3}{\mu}\frac{1+\beta}{\beta} L^4.
\end{align*}
We need to solve $p_{Z}(\lambda) = 0$, namely
\begin{align*}
 \lambda^2 &- (\gamma + (1+\beta)(\sigma+\alpha \frac{\sqrt {n-1}}{\sqrt{n}}L)^2)\lambda \cr
 &+ \gamma(1+\beta)(\sigma+\alpha \frac{\sqrt {n-1}}{\sqrt{n}}L)^2 \cr
 &-\gamma\frac{n-1}{n}\frac{2\alpha^3}{\mu}\frac{1+\beta}{\beta} L^4 =0.
\end{align*}

This results in 
\begin{align*}
 \lambda_{1,2} = \frac{\gamma + (1+\beta)(\sigma+\alpha\frac{\sqrt {n-1}}{\sqrt{n}} L)^2\pm \sqrt D}{2},
\end{align*}
where
\begin{align}\label{eq:discr}
 &D = (\gamma + (1+\beta)(\sigma+\alpha\frac{\sqrt {n-1}}{\sqrt{n}} L)^2)^2 \cr
 &- 4\left[\gamma(1+\beta)(\sigma+\alpha \frac{\sqrt {n-1}}{\sqrt{n}}L)^2 -\gamma\frac{n-1}{n}\frac{2\alpha^3}{\mu}\frac{1+\beta}{\beta} L^4\right]\cr
 &=(\gamma - (1+\beta)(\sigma+\alpha\frac{\sqrt {n-1}}{\sqrt{n}} L)^2)^2\cr
 &\qquad\qquad\qquad\qquad+8\gamma\frac{n-1}{n}\frac{\alpha^3}{\mu}\frac{1+\beta}{\beta} L^4.
\end{align}
We let $\lambda_1$ denote the largest (positive) eigenvalue.
Borrowing the idea from the proof of Lemma~7 and Lemma~17 in \cite{NestDistrOpt}, we notice that, due to properties of the diagonalization, any element of the matrix $Z^t$ is in the form $z^{ij}_t = z_t = c_1\lambda_1^t + c_2\lambda_2^t$, $i,j=1,2$, for some $c_1,c_2\in\mathbb C$ (we omit the element upper index in  $z_t^{ij}$ to simplify notations). To estimate $c_1,c_2$ for each element, we construct the following system of linear equalities:
\begin{align*}
\begin{cases}
 c_1 + c_2 &= z_0,\\
c_1 \lambda_1 + c_2 \lambda_2 &= z_1,
\end{cases}
\end{align*}
where $z_0$ and $z_1$ are the corresponding elements of the matrices $Z^0$ and $Z$ respectively. The solution of the system above is 
\begin{align}
 c_1 &= \frac{z_1 - z_0\lambda_2}{\lambda_1 - \lambda_2},\cr
 c_2 &= \frac{z_0\lambda_1 - z_1}{\lambda_1 - \lambda_2}.
\end{align}
Thus, 
\begin{align}\label{eq:ineq4}
 z_t& = \frac{z_1\lambda_1^t - z_0\lambda_2\lambda_1^t}{\lambda_1 - \lambda_2} + \frac{z_0\lambda_1\lambda_2^t - z_1\lambda_2^t }{\lambda_1 - \lambda_2}\cr
 &=\frac{z_1(\lambda_1^t-\lambda_2^t) -z_0\lambda_2\lambda_1(\lambda_1^{t-1} - \lambda_2^{t-1})}{\lambda_1 - \lambda_2}\cr
 &\le \frac{z_1(\lambda_1^t-\lambda_2^t)}{\lambda_1 - \lambda_2}\le 2z_1\frac{\lambda_1^t}{\lambda_1 - \lambda_2},
\end{align}
where in the last two inequalities we used Perron-Frobenius Theorem for positive matrices, namely $\lambda_1>|\lambda_2|$ (see Theorem~8.2.11 in \cite{MatrixAnalysis}).
As 
\[\|\bx^{t+1} - \bx^*\|^2_{\Fro}\le 2\|\bbx^{t+1} - \bx^*\|^2_{\Fro} +  2\|\bx^{t+1} - \bbx^{t+1}\|^2_{\Fro} ,\]
and by taking into account \eqref{eq:dynamics} and \eqref{eq:ineq4}, we conclude that 
\begin{align}\label{eq:distToOpt}
 \|\bx^{t+1} - \bx^*\|^2_{\Fro}&\le 2z_t^{11}\|\bbx^0 - \bx^*\|^2_{\Fro}+2z_t^{12}\|\bx^0 - \bbx^0\|^2_{\Fro}\cr
 & \quad + 2z_t^{21}\|\bbx^0 - \bx^*\|^2_{\Fro}+2z_t^{22}\|\bx^0 - \bbx^0\|^2_{\Fro}\cr
 &\le\lambda_1^t\frac{4}{\lambda_1 - \lambda_2}[ (z_1^{11}+z_1^{21})\|\bbx^0 - \bx^*\|^2_{\Fro} \cr
 &\qquad\qquad\qquad+ (z_1^{12}+z_1^{22})\|\bx^0 - \bbx^0\|^2_{\Fro}].
\end{align}
In \eqref{eq:distToOpt} $z_1^{ij}$ is the $ij$th element of the matrix $Z^1=Z$.
Thus, to get the result it suffices to demonstrate that $\lambda_1<1$.
Recall that
\begin{align}\label{eq:lambda1}
 \lambda_{1} = \frac{\gamma + (1+\beta)(\sigma+\frac{\sqrt {n-1}}{\sqrt{n}}\alpha L)^2 + \sqrt{D}}{2},
\end{align}
where
\begin{align*}
 D = (\gamma &- (1+\beta)(\sigma+\alpha\frac{\sqrt {n-1}}{\sqrt{n}} L)^2)^2\cr
 &+8\gamma\frac{n-1}{n}\frac{\alpha^3}{\mu}\frac{1+\beta}{\beta} L^4.
\end{align*}

Let us now fix $\beta = \frac12\left(\frac{1}{\sigma^2}-1\right)$.
As 
\begin{align*}
 \alpha&<\frac{\sigma}{2L}\frac{\sqrt {n}}{\sqrt{n-1}}\left(\frac{\sqrt 2}{\sqrt{1+\sigma^2}}-1\right) \cr
 &= \frac{1}{2L}\frac{\sqrt {n}}{\sqrt{n-1}}\left(\frac{1}{\sqrt{1+\beta}}-\sigma\right),
\end{align*}
we conclude that
\begin{align}\label{eq:ineq5}
 (1+\beta)&(\sigma+\alpha\frac{\sqrt {n-1}}{\sqrt{n}} L)^2 \cr
 &<  (1+\beta)\left(\sigma+\frac{1}{2}\left(\frac{1}{\sqrt{1+\beta}}-\sigma\right)\right)^2\cr
 &=\frac{1}{4}(\sigma\sqrt{1+\beta}-1)^2.
\end{align}
As $\alpha<\frac{n}{\mu}\left(\frac{8}{(\sqrt{1+\sigma^2}-\sqrt 2)^2}-1\right)=\frac{n}{\mu}\left(\frac{4}{(\sigma\sqrt{1+\beta}-1)^2}-1\right)$, we conclude that 
\begin{align}\label{eq:ineq6}
 \gamma = \frac{1}{\left(1+\frac{\mu\alpha}{n}\right)}\ge\frac{1}{4}(\sigma\sqrt{1+\beta}-1)^2.
\end{align}
Bringing \eqref{eq:ineq5} and \eqref{eq:ineq6} together, we obtain
\[(1+\beta)(\sigma+\alpha\frac{\sqrt {n-1}}{\sqrt{n}} L)^2 < \gamma\]
and, thus, from \eqref{eq:lambda1}
\begin{align}\label{eq:lambda11}
 \lambda_{1} &< \gamma  + \sqrt{\gamma\frac{2\alpha^3}{\mu}\frac{1+\beta}{\beta}\frac{n-1}{n} L^4} \cr
 &= \frac{n}{n+\mu\alpha}  + \sqrt{\frac{n-1}{n+\mu\alpha}\frac{2\alpha^3}{\mu}\frac{1+\sigma^2}{1-\sigma^2} L^4}.
\end{align}
Next, taking into account that $\alpha<\frac{\sqrt{n^2+\frac{2\mu^4(1-\sigma^2)}{(n-1)L^4(1+\sigma^2)}}-n}{2\mu}$ and \eqref{eq:lambda11}, we conclude that\footnote{More details can be found in Appendix.}
\[\lambda_{1}<1.\]
Finally, according to \eqref{eq:distToOpt},
\[\|\bx^{t+1} - \bx^*\|^2_{\Fro}\le O(q^t),\]
where
\begin{align}\label{eq:final_q}
q(\alpha) = \lambda_{1} = \frac{n}{2(n+\mu\alpha)} &+ \frac{(1+\sigma^2)(\sigma+\alpha\frac{\sqrt {n-1}}{\sqrt{n}} L)^2}{4\sigma^2} \cr
&\quad + \frac{\sqrt{D}}{2},
\end{align}
where
\begin{align*}
 D = &\left(\frac{n}{n+\mu\alpha} - \frac{1}{2}(1+\frac{1}{\sigma^2})(\sigma+\alpha\frac{\sqrt {n-1}}{\sqrt{n}} L)^2\right)^2 \cr
 &\qquad\qquad\qquad+ 8\frac{n-1}{n+\mu\alpha}\frac{\alpha^3}{\mu}\frac{1+\sigma^2}{1-\sigma^2} L^4.
\end{align*}
\end{proof}

\section{Comparison with the convergence rate of the GRANE}\label{sec:algorithms}
In this section we compare the convergence rate of the algorithm \eqref{eq:alg} analyzed in this paper and the convergence rate of the GRANE procedure studied in \cite{GRANE}, given some large number of players $n$.

According to Theorem 9 in \cite{GRANE} under Assumptions~\ref{assum:convex}-\ref{assum:connected} made above, the GRANE converges to the Nash equilibrium with the rate $O\left(\left(1-\frac{1}{\gamma_r^2}\right)^t\right)$, where $\gamma_r = \frac{L_{\bF_a}}{\mu_{r,\bF_a}}>1$ and the constants $L_{\bF_a}, \mu_{r,\bF_a}$ are defined in Lemmas 1 and 3 respectively\footnote{The constant $L_{\bF} = \max_i\{\sqrt{L_i^2+L_{-i}^2}\}$ defined in Lemma 3 in \cite{GRANE} corresponds to the constant $L=\max_i L_i$, where $L_i$s are defined in Assumption~\ref{assum:Lipschitz} in this paper.}. After substituting the expressions for $L_{\bF_a}, \mu_{r,\bF_a}$ into $\gamma_r$, we conclude that for a sufficiently large~$n$
\[\gamma_r = 2n\left[\frac{L}{\mu} + \frac{\sigma_{\max}\{I-W\}}{\alpha^0 \mu}\right].\]
Next, according to Remark 4 in \cite{GRANE}, 
\[\alpha^0<\frac{\lminnz{I-W}}{L\left(1+\frac{1}{\beta^2}\right)}, \quad \mbox{where $\beta\sim\frac{\mu}{nL}$. }\]

Thus, given the optimal choice of $\alpha^0$, we get 
\[\gamma_r = 2n\left[\frac{L}{\mu} + \frac{L/{\mu}\left(1+\frac{n^2L^2}{\mu^2}\right)\sigma_{\max}\{I-W\}}{\lminnz{I-W}}\right].\]
Thus, the convergence rate of the GRANE is 
\begin{align}\label{eq:RateGRANE}
 O\left(\left(1-\frac{\mu^6}{L^6n^6}\right)^t\right).
\end{align}

Now we proceed with the convergence rate estimation of the algorithm \eqref{eq:alg}.
According to the proof of Theorem~\ref{th:main}, the convergence rate of the distributed procedure is $O(q(\alpha)^t)$, where (see \eqref{eq:lambda11}) 
\[q(\alpha)< \frac{n}{n+\mu\alpha}  + \sqrt{\frac{n-1}{n+\mu\alpha}\frac{2\alpha^3}{\mu}\frac{1+\sigma^2}{1-\sigma^2} L^4}.\]
The constant $q$ above is less than $1$, if 
\begin{align*}
\alpha&<\frac{\sqrt{n^2+\frac{2\mu^4(1-\sigma^2)}{(n-1)L^4(1+\sigma^2)}}-n}{2\mu} \cr
&= \frac{n}{2\mu}\left(1+\frac{2\mu^4(1-\sigma^2)}{(n-1)n^2L^4(1+\sigma^2)}\right)^{1/2}-\frac{n}{2\mu}\cr
&\sim\frac{\mu^3(1-\sigma^2)}{2n(n-1)L^4(1+\sigma^2)}. 
\end{align*}
Thus, taking into account two inequalities above, we conclude that for a sufficiently large $n$
\begin{align}\label{eq:rateNEW}
 q(\alpha\sim&\frac{\mu^3(1-\sigma^2)}{2n(n-1)L^4(1+\sigma^2)})< O\left(\frac{n}{n+\mu\alpha}\right) \cr&= O\left((1+\frac{\mu\alpha}{n})^{-1}\right)= O\left(1-\frac{\mu\alpha}{n}\right)\cr
 &=O\left(1-\frac{\mu^4}{L^4n^2(n-1)}\right).
\end{align}

Next, let us notice that, under Assumption~\ref{assum:Lipschitz}, the game mapping $\bF$ defined in \eqref{eq:gamemapping} is Lipschitz continuous with the constant $L^{\bF} = L\sqrt n$. Indeed, due to Assumption~\ref{assum:Lipschitz}, 
\begin{align*}
\|\bF(x)-\bF(y)\|^2 &= \sum_{i=1}^n(\nabla_i J_i(x) - \nabla_i J_i(y))^2\cr
&\le \sum_{i=1}^nL^2_i\|x - y\|^2\le nL^2\|x - y\|^2.
\end{align*}
Thus, the condition number of the mapping $\bF$ is 
\begin{align}\label{eq:condNumber}
 \frac{L^{\bF}}{\mu} = \frac{L\sqrt n}{\mu}\ge1.
\end{align}

By comparing \eqref{eq:RateGRANE} and \eqref{eq:rateNEW} and taking into account \eqref{eq:condNumber}, we conclude that the convergence rate of the proposed algorithm \eqref{eq:alg} is faster than that of the GRANE presented in \cite{GRANE}.

\section{Simulation}\label{sec:sim}
Let us consider a class of games  with strongly monotone game mappings. Specifically, we have players $\{1,2,\ldots,20\}$ and each player $i$'s objective is to minimize the cost function $J_i(x_i,x_{-i})=f_i(x_i)+l_i(x_{-i})x_i$, where $f_i(x_i)=0.5a_ix_i^2+b_ix_i$ and $l_i(x_{-i})=\sum_{j\neq i}c_{ij}x_j$. The local cost function is in general dependent on actions of all players, but the underlying communication graph is a randomly generated tree graph. We randomly select $a_i$, $b_i$, and $c_{ij}$ for all possible $i$ and $j$. 

We simulate the proposed gradient play algorithm and compare its implementation with the implementation of the algorithm GRANE presented in \cite{GRANE} (see Figure~\ref{eps:compare}). The GRANE is based on a so called augmented game mapping, for which an additional parameter
has to be chosen to guarantee specific properties of this mapping and, thus, convergence of the procedure. Note that the GRANE is very sensitive to the setting of this parameter.
We chose this parameter based on the theoretic results in \cite{GRANE}. For the gradient play we chose the step size parameter $\alpha$ based on Theorem~\ref{th:main} (in the presented simulation $\alpha=0.05$). As we can see, the gradient play outperforms the GRANE. 
\begin{figure}[!htb]
	\centering
	\psfrag{0}[c][l]{\small$0$}
	\psfrag{1}[c][c]{\small$1$}
	\psfrag{0.8}[c][c]{\small$0.8$}
	\psfrag{0.6}[c][c]{\small$0.6$}
	\psfrag{0.4}[c][c]{\small$0.4$}
	\psfrag{0.2}[c][c]{\small$0.2$}
	\psfrag{200}[c][c]{\small$200$}
	\psfrag{400}[c][c]{\small$400$}
	\psfrag{600}[c][c]{\small$600$}
	\psfrag{800}[c][c]{\small$800$}
	\psfrag{1000}[c][c]{\small$1000$}
	\psfrag{R}[c][l]{\small{Relative Error}}
	\psfrag{t}[c][c]{time}
        \includegraphics[width=0.5\textwidth]{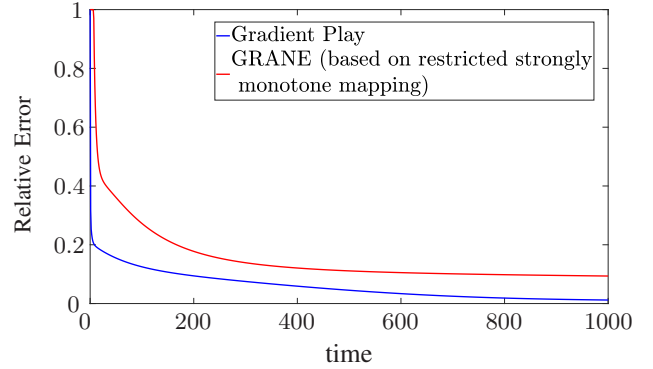}
        \caption{\label{eps:compare}Comparison of the presented algorithm and GRANE based on restricted strongly monotone augmented mapping}
        \end{figure}
        
\section{conclusion}\label{sec:conclusion}
In this paper, we have presented the distributed gradient play which provably converges to a Nash equilibrium in strongly convex games with unconstrained action sets. In comparison to the GRANE algorithm \cite{GRANE}, which possesses a geometric convergence rate as well, the proposed algorithm requires only one parameter (the step size) to be appropriately choose. Moreover, its convergence rate is shown to be faster given some fixed game under consideration. 
The future work can be devoted to investigation of the convergence rate of the gradient projected play in the case of bounded closed agents' action sets.
        
\bibliographystyle{plain}
\bibliography{document}

\vspace{1cm}
\textbf{Appendix}

Here we demonstrate that $\frac{n}{n+\mu\alpha}  + \sqrt{\frac{n-1}{n+\mu\alpha}\frac{2\alpha^3}{\mu}\frac{1+\sigma^2}{1-\sigma^2} L^4}<1$ (see \eqref{eq:lambda11}) if $0<\alpha<\frac{\sqrt{n^2+\frac{2\mu^4(1-\sigma^2)}{(n-1)L^4(1+\sigma^2)}}-n}{2\mu}$. 

Indeed, 
$$\frac{n}{n+\mu\alpha}  + \sqrt{\frac{n-1}{n+\mu\alpha}\frac{2\alpha^3}{\mu}\frac{1+\sigma^2}{1-\sigma^2} L^4}<1$$
$$\Updownarrow$$
$$\sqrt{(n-1)\frac{2\alpha}{\mu^3}\frac{1+\sigma^2}{1-\sigma^2} L^4}<\frac{1}{\sqrt{n+\mu\alpha}}$$
$$\Updownarrow$$
$$\alpha c(n+\mu\alpha)-1 = c\mu\alpha^2+nc\alpha-1<0,$$
where $c = \frac{2(n-1)}{\mu^3}\frac{1+\sigma^2}{1-\sigma^2} L^4$.
The last inequality above holds, if 
$$0<\alpha<\frac{-n+\sqrt{n^2+\frac{4\mu}{c}}}{2\mu}.$$
By substituting $c$ in the expression above, we obtain 
$$0<\alpha<\frac{\sqrt{n^2+\frac{2\mu^4(1-\sigma^2)}{(n-1)L^4(1+\sigma^2)}}-n}{2\mu}.$$
\end{document}